\documentclass{amsart}
\usepackage{latexsym}
\usepackage{amsfonts}
\usepackage{amssymb}
\usepackage{amsmath}
\usepackage{mathrsfs}
\usepackage{wasysym}
\usepackage{hyperref}
\usepackage{textcomp}
\usepackage{tikz}
\usetikzlibrary{matrix,arrows}

\title[Comparison between two differential graded algebras in NCG]{Comparison between two differential graded algebras\\
in Noncommutative Geometry}
\author{Partha Sarathi Chakraborty}
\address{The Institute of Mathematical Sciences, CIT Campus, Taramani, Chennai
600113}
\email{parthac@imsc.res.in}

\author{Satyajit Guin}
\address{Indian Institute of Science Education and Research, Mohali, Punjab 140306}
\email{satyamath@gmail.com , satyajit@iisermohali.ac.in}

\thanks{$\dagger$ Satyajit Guin is supported by INSPIRE Faculty Award (Award No. DST/INSPIRE/04/2015/000901)}

\date{\today}

\newcommand{\hooklongrightarrow}{\lhook\joinrel\longrightarrow}
\newcommand{\twoheadlongrightarrow}{\relbar\joinrel\twoheadrightarrow}

\newtheorem{definition}{Definition}[section]
\newtheorem{theorem}[definition]{Theorem}
\newtheorem{lemma}[definition]{Lemma}
\newtheorem{proposition}[definition]{Proposition}
\newtheorem{corollary}[definition]{Corollary}

\newtheorem{remark}[definition]{Remark}

\begin{document}

\begin{abstract}
Starting with a spectral triple one can associate two canonical differential graded algebras (dga) defined by Connes and Fr\"{o}hlich et\ al. For
the classical spectral triples associated with compact Riemannian spin manifolds both these dgas coinside with the de-Rham dga. Therefore, both are
candidates for the noncommutative space of differential forms. Here we compare these two dgas and observe that in a very precise sense Connes' dga
is more informative than that of Fr\"{o}hlich et al.
\end{abstract}

\maketitle

\section{Introduction}
A differential calculus on a ``space'' means the specification of a differential graded algebra (dga), often interpreted as space of forms. In
classical geometry the ``space'' is a manifold and we have the de-Rham dga, whereas in noncommutative geometry a ``space'' is described by a triple
called spectral triple. A spectral triple is a tuple $(\mathcal{A},\mathcal{H},D)$ where $\mathcal{A}$ is an associative $\,\star$-algebra represented
on the Hilbert space $\mathcal{H}$ and $D$ is a Dirac-type operator on $\mathcal{H}$. Associated to a spectral triple there are two canonical dgas
defined by Connes (\cite{con}) and Fr\"{o}hlich et\ al (\cite{fgr}). In literature, these are denoted by $\,\Omega_D^\bullet(\mathcal{A})$ and
$\widetilde{\Omega}_D^\bullet(\mathcal{A})$ respectively, and here we call them as the Dirac dga and the FGR dga. Note that in (\cite{cg}) we have
called the Dirac dga as the Connes' calculus. It should be noted that for the classical spectral triple associated with compact Riemannian spin
manifolds both these dgas coincide with the de-Rham dga (\cite{con},\cite{fgr}). Therefore, both are candidates to be declared as noncommutative
space of forms. Moreover, they are same for the noncommutative torus (Page $172$ in \cite{fgr}) but not for the $SU_q(2)$ (\cite{cp}). Hence, it is
natural to ask if there is any way to compare these two dgas so that one can declare one of them as truely the noncommutative space of forms. This
is important because both being generalization of the classical de-Rham forms to the noncommutative set up, any notion in noncommutative geometry
involving the noncommutative space of forms, e.g. the Yang-Mills functional (\cite{con}), can be defined using either the Dirac dga or the FGR dga.
Hence, a comparison is needed to overcome the difficulty of choice between these two dgas. This is precisely the goal of our investigation in this
article. Main conclusion of this article is ``\textbf{Dirac dga is more informative than that of Fr\"{o}hlich et al.}'' and our task is substantiating
this claim. Precise meaning of ``more informative'' is given through explicit computation of both these dgas for a family of spectral triples. In
the literature, these have been computed in very few cases like noncommutative torus, $SU_q(2)$. This indicates that probably these are difficult to
compute and we had no clue on how to compare them. Recently, authors have identified suitable hypotheses which allow the computation of the Dirac dga
$\,\Omega_D^\bullet$ for a class of spectral triples. This gives the first systematic computation of $\,\Omega_D^\bullet$ for a large family of
spectral triples (\cite{cg}). In this article we compute the FGR dga $\,\widetilde{\Omega}_D^\bullet$ for the same family of spectral triples, and
this leads to a comparison between these two dgas.

To describe our computation in detail we recall the concept of the quantum double suspension (QDS) of a $C^*$-algebra $\mathcal{A}$, denoted by
$\varSigma^2 \mathcal{A}\,$, introduced by Hong-Szymanski in (\cite{hsz}). Later QDS of a spectral triple was introduced by Chakraborty-Sundar
(\cite{csu}). We record here few significance of QDS.

\textbf{Significance of QDS~:}
\begin{itemize}
\item[(a)] Quantum even and odd dimensional spheres are produced by iterating QDS to two points and the circle, respectively (\cite{hsz}).
\item[(b)] Noncommutative analogues of n-dimensional balls are obtained by repeated application of the QDS to the classical low-dimensional spaces (\cite{hsz1}).
\item[(c)] If we have one spectral triple $(\mathcal{A},\mathcal{H},D)$ then iterating QDS we produce many spectral triples. Thus, iterating QDS on
the classical cases of manifolds one produces genuine noncommutative spectral triples. Moreover, finite summability, $\varTheta$-summability,
even-ness all are preserved under the iteration.
\item[(d)] All the torus-equivariant spectral triples on the odd dimensional quantum spheres are obtained by iterating QDS to the spectral triple
$\left(C^\infty(S^1),L^2(S^1),-i\frac{d}{d\theta}\right)$.
\item[(e)] Most importantly, QDS produces a class of examples of regular spectral triples having simple dimension spectrum (\cite{csu}), essential
in the context of local index formula of Connes-Moskovici (\cite{cmo}).
\end{itemize}
This article adds one more significance to the above list namely, QDS provides a comparison between the Dirac dga and the FGR dga and establishes
the Dirac dga as more appropritae generalization of the classical de-Rham dga to the noncommutative set-up. We work here under the following mild
hypotheses on a spectral triple $(\mathcal{A},\mathcal{H},D)~$:
\begin{itemize}
\item[1.] $[D,a]F-F[D,a]$ is a compact operator for all $a\in\mathcal{A}$, where $F$ is the sign of the operator $D$,
\item[2.] $\mathcal{H}^\infty:=\bigcap_{k\geq 1}\mathcal{D}om(D^k)$ is a left $\mathcal{A}$-module, and $[D,\mathcal{A}]\subseteq\mathcal{A}\otimes
\mathcal{E}nd_{\mathcal{A}}(\mathcal{H}^\infty)\subseteq\mathcal{E}nd_{\mathbb{C}}(\mathcal{H}^\infty)$.
\end{itemize}
Notable features of these hypotheses are firstly, the spectral triple associated with a first order differential operator on a manifold will always
satisfy them and secondly, they are stable under the quantum double suspension. The authors have computed $\,\Omega_D^\bullet$ for the quantum
double suspended spectral triple $(\varSigma^2 \mathcal{A},\varSigma^2 \mathcal{H},\varSigma^2 D)$ in (\cite{cg}) under these conditions. It turns
out that the FGR dga becomes almost trivial for $(\varSigma^2 \mathcal{A},\varSigma^2 \mathcal{H},\varSigma^2 D)$ in the sense that it does not
reflect any information about $(\mathcal{A},\mathcal{H},D)$. This phenomenon was observed in (\cite{cp}) for the $SU_q(2)$. Since, the torus
equivariant spectral triples on the odd dimensional quantum spheres are obtained through iterated QDS on the spectral triple $(C^\infty(S^1),
L^2(S^1),-i\frac{d}{d\theta})$, this article also extends earlier work of Chakraborty-Pal (\cite{cp}). This helps us to conclude, in view of
(\cite{cg}), that the Dirac dga is more informative than the FGR dga.

Organization of this paper is as follows. In Section ($2$) we discuss Dirac dga $\,\Omega_D^\bullet$, the quantum double suspension and obtain few
results. Section ($3$) mainly deals with the computation of the FGR dga $\,\widetilde \Omega_{\varSigma^2D}^\bullet(\varSigma^2\mathcal{A})$ for
QDS, which finally leads us to the comparison between Connes' and FGR dga.
\bigskip

%%%%%%%%%%%%%%%%%%%%%%%%%%%%%%%  Dirac dga & the Quantum Double Suspension  %%%%%%%%%%%%%%%%%%%%%%%%%%%%%%%%%%%%%%%%%%%%%%%%%%%%%%%%%%%%%%%%%%%

\section{Dirac DGA and The Quantum Double Suspension}

In this section we recall the definition of Dirac dga $\,\Omega_D^\bullet\,$ from (\cite{con}), and the quantum double suspension from
(\cite{hsz},\cite{csu}).

\begin{definition}
A {\it spectral triple} $(\mathcal{A},\mathcal{H},D)$ over an involutive associative algebra $\mathcal{A}$ consists of the following things $:$
\begin{enumerate}
\item a $\,\star\,$-representation $\pi$ of $\mathcal{A}$ on a Hilbert space $\mathcal{H}$,
\item an unbounded selfadjoint operator $D$ acting on $\mathcal{H}$,
\item $D$ has compact resolvent and $[D,a]$ extends to a bounded operator on $\mathcal{H}$ for every $\,a \in \mathcal{A}$.
\end{enumerate}
\end{definition}

We will assume that $\mathcal{A}$ is unital and $\pi$ is a unital representation. If $|D|^{-p}$ is in the ideal of Dixmier traceable operators
$\mathcal{L}^{(1,\infty)}$ then we say that the spectral triple is $\,p$-summable. In literature, this is sometimes denoted by $\,p^+$-summable,
$(p,\infty)$-summable etc. Moreover, if there is a $\mathbb{Z}_2$-grading $\gamma\in\mathcal{B}(\mathcal{H})$ such that $\gamma$ commutes with
every element of $\mathcal{A}$ and anticommutes with $D$ then the spectral triple $(\mathcal{A},\mathcal{H},D,\gamma)$ is said to be an
{\it even spectral triple}. Associated to every spectral triple we have the following differential graded algebra (dga).

\begin{definition}[\cite{con},\cite{cg}]\label{our defn of Connes}
Let $(\mathcal{A},\mathcal{H},D)$ be a spectral triple and $\,\Omega^\bullet(\mathcal{A})=\bigoplus_{k=0}^\infty\Omega^k(\mathcal{A})\,$ be the
reduced universal differential graded algebra over $\mathcal{A}$. Here, $\,\Omega^k(\mathcal{A}):=span\{a_0da_1\ldots da_k\colon a_i\in\mathcal{A},
i=1,\ldots,k\},\,d$ being the universal differential. With the convention $(da)^*=-da^*$, we get a $\star\,$-representation $\pi$ of $\,\Omega^\bullet
(\mathcal{A})$ on $\mathcal{Q}(\mathcal{H}):=\mathcal{B}(\mathcal{H})/\mathcal{K}(\mathcal{H})$, given by
\begin{center}
$\pi(a_0da_1\ldots da_k) := a_0[D,a_1]\ldots [D,a_k]+\mathcal{K}(\mathcal{H})\, \, ;\, \, a_j \in \mathcal{A}\,$.
\end{center}
Let $\,J_0^{(k)}=\{\omega\in\Omega^k:\pi(\omega)=0\}$ and $\,J^\prime=\bigoplus J_0^{(k)}$. Since $\,J^\prime$ fails to be a differential ideal in
$\,\Omega^\bullet$ consider $J^\bullet=\bigoplus J^{(k)}$, where $\,J^{(k)}=J_0^{(k)}+dJ_0^{(k-1)}$. Then $\,J^\bullet$ becomes a differential graded
two-sided ideal in $\,\Omega^\bullet$ and hence, the quotient $\,\Omega_D^\bullet=\Omega^\bullet/J^\bullet$ becomes a differential graded algebra,
called the Connes' calculus or the Dirac dga.
\end{definition}

The representation $\pi$ gives the following isomorphism
\begin{eqnarray}\label{main iso}
\Omega_D^k \cong \pi(\Omega^k)/\pi(dJ_0^{k-1})\,,\quad\forall\,k\geq 1.
\end{eqnarray}
The differential $d$ on $\Omega^\bullet(\mathcal{A})$ induces a differential, denoted again by $d$, on the complex $\Omega^\bullet_D(\mathcal{A})$
so that we get a chain complex $(\,\Omega^\bullet_D(\mathcal{A}),d\,)$ and a chain map $\,\pi_D:\Omega^\bullet(\mathcal{A})\longrightarrow
\Omega^\bullet_D(\mathcal{A})$ such that the following diagram 
\begin{center}
\begin{tikzpicture}[node distance=3cm,auto]
\node (Up)[label=above:$\pi_D$]{};
\node (A)[node distance=1.5cm,left of=Up]{$\Omega^\bullet(\mathcal{A})$};
\node (B)[node distance=1.5cm,right of=Up]{$\Omega^\bullet_D(\mathcal{A})$};
\node (Down)[node distance=1.5cm,below of=Up, label=below:$\pi_D$]{};
\node(C)[node distance=1.5cm,left of=Down]{$\Omega^{\bullet+1}(\mathcal{A})$};
\node(D)[node distance=1.5cm,right of=Down]{$\Omega^{\bullet+1}_D(\mathcal{A})$};
\draw[->](A) to (B);
\draw[->](C) to (D);
\draw[->](B)to node{{ $d$}}(D);
\draw[->](A)to node[swap]{{ $d$}}(C);
\end{tikzpicture} 
\end{center}
commutes. Note that $\,\Omega_D^\bullet(\mathcal{A})$ can be defined for non-unital algebra $\mathcal{A}$ as well as prescribed in (\cite{cg},
after Remark $[2.3]$).

\begin{lemma}\label{filtered algebra}
If there is a decreasing filtration $$\mathcal{A}=\mathcal{A}_0\supseteq\mathcal{A}_{-1}\supseteq\ldots\ldots\supseteq\{0\}$$ of subspaces of
$\mathcal{A}$ then $\,\Omega_D^\bullet(\mathcal{A})$ becomes a filtered algebra.
\end{lemma}
\begin{proof}
Let $\,J_0^{k,n}=ker\left(\pi^k|_{\Omega^k(\mathcal{A}_n)}\right)$. Then $\,J_0^{k,n}\subseteq J_0^{k,n+1}$. If we let $\,J^{k,n}=J_0^{k,n}+
dJ_0^{k-1,n}$ then $\,J^{k,n}\subseteq J^{k,n+1}$. We have $$\Phi^{k,n}:\frac{\Omega^k(\mathcal{A}_n)}{J^{k,n}}\hooklongrightarrow\frac{\Omega^k
(\mathcal{A}_{n+1})}{J^{k,n}}\twoheadlongrightarrow\frac{\Omega^k(\mathcal{A}_{n+1})}{J^{k,n+1}}$$ with
\begin{eqnarray*}
Ker(\Phi^{k,n}) & = & \{\omega\in\Omega^k(\mathcal{A}_n):\omega\in J^{k,n+1}\}\\
& = & J^{k,n+1}\cap\Omega^k(\mathcal{A}_n)/J^{k,n}\,,
\end{eqnarray*}
and
\begin{eqnarray*}
\mathcal{I}m(\Phi^{k,n}) & = & \Omega^k(\mathcal{A}_n)/\Omega^k(\mathcal{A}_n)\cap J^{k,n+1}\,.
\end{eqnarray*}
This gives a filtration on $\,\Omega_D^\bullet(\mathcal{A})$.
\end{proof}

\begin{proposition}\label{associated graded algebra}
The associated graded algebra of the filtered algebra $\,\Omega_D^\bullet(\mathcal{A})$ is given by $$\mathcal{G}=\bigoplus_{n\leq 0}
\bigoplus_{p\geq 0}\frac{\Omega^p(\mathcal{A}_n)}{\Omega^p(\mathcal{A}_{n-1})+J^{p,n}}\,.$$
\end{proposition}
\begin{proof}
By Lemma (\ref{filtered algebra}), the filtration on $\,\Omega_D^\bullet(\mathcal{A})$ is given by $\,\mathcal{F}_n=\bigoplus_{k\geq 0}\Omega^k
(\mathcal{A}_n)/J^{k,n}$. Hence, the associated graded algebra is given by $\,\mathcal{G}=\bigoplus_{n\leq 0}\mathcal{G}_n$ where,
\begin{eqnarray*}
\mathcal{G}_n & = & \mathcal{F}_n/\mathcal{F}_{n-1}\\
& = & \frac{\bigoplus_{p\geq 0}\Omega^p(\mathcal{A}_n)/J^{p,n}}{\bigoplus_{q\geq 0}\Omega^q(\mathcal{A}_{n-1})/J^{q,n-1}}\\
& = & \bigoplus_{p\geq 0}\frac{\Omega^p(\mathcal{A}_n)/J^{p,n}}{\Omega^p(\mathcal{A}_{n-1})/J^{p,n-1}}\\
& = & \bigoplus_{p\geq 0}\frac{\Omega^p(\mathcal{A}_n)/J^{p,n}}{\mathcal{I}m(\Phi^{p,n-1})}\\
& = & \bigoplus_{p\geq 0}\frac{\Omega^p(\mathcal{A}_n)/J^{p,n}}{\Omega^p(\mathcal{A}_{n-1})/\Omega^p(\mathcal{A}_{n-1})\cap J^{p,n}}\\
& = & \bigoplus_{p\geq 0}\frac{\Omega^p(\mathcal{A}_n)}{\Omega^p(\mathcal{A}_{n-1})+J^{p,n}}
\end{eqnarray*}
\end{proof}

Now we define the quantum double suspension (QDS) of $C^*$-algebras and spectral triples.
\medskip

\textbf{Notation:~}
\begin{enumerate}
\item We denote by `$l$' the left shift operator on $\ell^2(\mathbb{N})$, defined on the standard orthonormal basis $(e_n)$ by $\,l(e_n) = e_{n-1}$,
$l(e_0) = 0$.
\item `$N$' be the number operator on $\ell^2(\mathbb{N})$ defined by $N(e_n) = ne_n$.
\item `$u$' denotes the rank one projection $|e_0\rangle \langle e_0| := I - l^*l\,$.
\item $\mathcal{K}$ denotes the space of compact operators on $\ell^2(\mathbb{N})$.
\end{enumerate}

\begin{definition}[\cite{hsz}]
Let $\mathcal{A}$ be a unital $C^*$-algebra. The quantum double suspension of $\mathcal{A}$, denoted by $\varSigma^2{\mathcal{A}}$, is the
$C^*$-algebra generated by $a\otimes u$ and $1\otimes l$ in $\mathcal{A}\otimes \mathscr{T}$, where $\mathscr{T}$ is the Toeplitz algebra.
\end{definition}

There is a symbol map $\sigma:\mathscr{T}\longrightarrow C(S^1)$ which sends $l$ to the standard unitary generator $\,z$ of $C(S^1)$ and one
gets the following short exact sequence
\begin{center}
$0\longrightarrow\mathcal{K}\longrightarrow\mathscr{T}\stackrel{\sigma}\longrightarrow C(S^1)\longrightarrow 0\,$.
\end{center}
If $\rho$ denotes the restriction of $1\otimes\sigma$ to $\varSigma^2\mathcal{A}$ then one has the following short exact sequence
\begin{center}
$0\longrightarrow\mathcal{A}\otimes\mathcal{K}\longrightarrow\varSigma^2\mathcal{A}\stackrel{\rho}\longrightarrow C(S^1)\longrightarrow 0\,$.
\end{center}
There is a $\mathbb{C}\,$-linear splitting map $\sigma^\prime$ from $C(S^1)$ to $\varSigma^2\mathcal{A}$ which sends the standard unitary generator
$z$ of $C(S^1)$ to $1\otimes l$, and yields the following $\mathbb{C}$-vector spaces (not as algebras) isomorphism $:$
\begin{center}
$\varSigma^2\mathcal{A}\cong\left(\mathcal{A}\otimes\mathcal{K}\right)\bigoplus C(S^1)\,.$
\end{center}
Notice that $\sigma^\prime$ is injective since it has a left inverse $\,\rho\,$ and hence, any $f\in C(S^1)$ can be identified with $1\otimes
\sigma^\prime(f)\in\varSigma^2\mathcal{A}$. For $f=\sum_n\lambda_nz^n\in C(S^1)$, we write $\,\sigma^{\prime}(f):=\sum_{n\geq 0}\lambda_nl^n+
\sum_{n>0}\lambda_{-n}l^{*n}$. Now let $\mathcal{A}$ be a dense $\star$-subalgebra of a $C^*$-algebra $\mathbb{A}$. Define
\begin{center}
$\varSigma^2_{alg}\mathcal{A}:=span\{a\otimes T,1\otimes l^m,1\otimes(l^*)^n\colon a\in\mathcal{A},T\in\mathbb{S}(\ell^2(\mathbb{N})),m,n\geq 0\}$
\end{center}
where, $\,\mathbb{S}(\ell^2(\mathbb{N})):=\{T=(\alpha_{ij})\colon\sum_{i,j}(1+i+j)^k|\alpha_{ij}|<\infty\,\,\forall\,k\geq 0\}$ is the space of
Schwartz class operators on $\ell^2(\mathbb{N})$. Clearly, $\,\varSigma^2_{alg}\mathcal{A}$ is a dense subalgebra of $\,\varSigma^2\mathbb{A}$ and
we have the following $\mathbb{C}$-vector spaces (not as algebras) isomorphism at the subalgebra level $:$
\begin{center}
$\varSigma^2_{alg}\mathcal{A}\cong\left(\mathcal{A}\otimes\mathbb{S}(\ell^2(\mathbb{N}))\right)\bigoplus\mathbb{C}[z,z^{-1}]\,.$
\end{center}

\begin{definition}[\cite{csu}]
For any spectral triple $(\mathcal{A},\mathcal{H},D),\,(\varSigma^2_{alg}\mathcal{A},\varSigma^2 \mathcal{H}:=\mathcal{H}\otimes \ell^2(\mathbb{N}),
\varSigma^2D:=D\otimes I+F\otimes N)$ becomes a spectral triple, where $F$ is the sign of the operator $D$ and $N$ is the number operator on
$\ell^2(\mathbb{N})$. This is called the quantum double suspension of the spectral triple $(\mathcal{A},\mathcal{H},D)$.
\end{definition}

It is easy to see that if $(\mathcal{A},\mathcal{H},D)$ is $p$-summable then $(\varSigma^2_{alg}\mathcal{A},\varSigma^2 \mathcal{H},\varSigma^2 D)$
is a $(p+1)$-summable spectral triple. Notice that for any $f\in\mathbb{C}[z,z^{-1}]\,$ we have $\,[\varSigma^2 D,1\otimes\sigma^\prime(f)]=F\otimes
[N,f]$. The finite subalgebra $(\varSigma^2_{alg}\mathcal{A})_{fin}$ is generated by $a\otimes T$ and $\sum_{0\leq n<\infty}\lambda_nl^n+\sum_{0<
n<\infty}\lambda_{-n}l^{*n}$, where $a\in\mathcal{A}$ and $T\in\mathcal{B}\left(\ell^2(\mathbb{N})\right)$ is a finitely supported matrix.

\begin{remark}\rm
In $($\cite{con}$)$, Connes represented $\,\Omega^\bullet(\mathcal{A})$ on $\mathcal{B}(\mathcal{H})$ instead on $\mathcal{Q}(\mathcal{H})$. But
the explicit computation of $\,\Omega_{\varSigma^2D}^\bullet((\varSigma^2_{alg}\mathcal{A})_{fin})$ is very difficult, even in the particular cases.
In $($\cite{cg}$)$ authors have computed $\,\Omega_{\varSigma^2D}^\bullet((\varSigma^2_{alg}\mathcal{A})_{fin})$ following the prescription given in
Definition $($\ref{our defn of Connes}$)$. Justification for this is also discussed in $($\cite{cg}$)$.
\end{remark}

The computation of $\,\Omega_{\varSigma^2D}^\bullet((\varSigma^2_{alg}\mathcal{A})_{fin})$ has been done in (\cite{cg}) under the following
conditions on spectral triples $(\mathcal{A},\mathcal{H},D)$.
\medskip

\textbf{Conditions}~:\\
$\hspace*{.7cm}(A)\,\,[D,a]F-F[D,a]$ is a compact operator for all $a\in\mathcal{A}\,$, where $F=sign(D)$.\\
$\hspace*{.7cm}(B)\,\,\mathcal{H}^\infty:=\bigcap_{k\geq 1}\mathcal{D}om(D^k)$ is a left $\mathcal{A}$-module and $\,[D,\mathcal{A}]\subseteq
\mathcal{A}\otimes\mathcal{E}nd_{\mathcal{A}}(\mathcal{H}^\infty)\subseteq\mathcal{E}nd_{\mathbb{C}}(\mathcal{H}^\infty)\,$.
\medskip

Notable features of these conditions are given by the following Proposition.

\begin{proposition}[\cite{cg}]
These conditions are valid for the classical case where $\mathcal{A}=C^\infty(\mathbb{M})$ and $D$ is a first order differential operator. Moreover,
if a spectral triple $(\mathcal{A},\mathcal{H},D)$ satisfies these conditions then the quantum double suspended spectral triple $(\varSigma^2_{alg}
\mathcal{A},\varSigma^2\mathcal{H},\varSigma^2D)$ also satisfies them.
\end{proposition}

\textbf{Notation}~: \begin{enumerate}
                     \item In this article we will work with $(\varSigma^2_{alg}\mathcal{A})_{fin}$ and denote it by $\,\varSigma^2 \mathcal{A}\,$
                     for notational brevity.
                     \item For all $f\in\mathbb{C}[z,z^{-1}],$ we denote $\,[N,f]$ by $\,f^\prime\,$ for notational brevity.
                     \item `$\mathcal{S}$' denotes the space of finitely supported matrices in $\mathcal{B}(\ell^2(\mathbb{N}))\, $.
                     \item $(e_{ij})$ will denote infinite matrix with $1$ at the $ij$-th place and zero elsewhere. We call it elementary matrix.
                    \end{enumerate}
\medskip

The notion of unitary equivalence of spectral triples forms a category of spectral triples. That is, we have the following.
\begin{definition}\label{category}
The objects of the category $\mathcal{S}pec$ are spectral triples $(\mathcal{A},\mathcal{H},D)$. A morphism between two such objects $(\mathcal{A}_i,
\mathcal{H}_i,D_i),i=1,2,$ is a tuple $(\phi,\Phi)$, where $\phi:\mathcal{A}_1\rightarrow\mathcal{A}_2$ is a unital algebra morphism and $\Phi:
\mathcal{H}_1\rightarrow\mathcal{H}_2$ is a unitary which intertwines the algebra representations and the Dirac operators $D_1,D_2$. 
\end{definition}

\begin{proposition}\label{Connes is functor}
The association $\mathcal{F}:(\mathcal{A},\mathcal{H},D)\longmapsto\Omega_D^\bullet(\mathcal{A})$ gives a covariant functor from $\mathcal{S}pec$ to
$DGA$, the category of differential graded algebras over $\mathbb{C}$.
\end{proposition}
\begin{proof}
Consider two objects $(\mathcal{A}_1,\mathcal{H}_1,D_1),(\mathcal{A}_2,\mathcal{H}_2,D_2)\in \mathcal{O}b(\mathcal{S}pec)$ and suppose there is a
morphism $(\phi\,,\Phi):(\mathcal{A}_1,\mathcal{H}_1,D_1)\longrightarrow (\mathcal{A}_2,\mathcal{H}_2,D_2)$. Define
\begin{align*}
\Psi:\Omega_{D_1}^\bullet(\mathcal{A}_1) &\longrightarrow \Omega_{D_2}^\bullet(\mathcal{A}_2)\\
\left[\sum a_0\prod_{i=1}^n[D_1,a_i]\right] &\longmapsto\left[\sum \phi(a_0)\prod_{i=1}^n[D_2,\phi(a_i)]\right]
\end{align*}
for all $\,a_j\in \mathcal{A}_1,\,n\geq 0\,$. To show $\Psi$ is well-defined we must show that $\Psi(\pi(d_1J_0^m))\subseteq \pi(d_2J_0^m)$
for all $m\geq 1$, where $d_1,d_2$ are the universal differentials for $\Omega^\bullet(\mathcal{A}_1),\Omega^\bullet(\mathcal{A}_2)$ respectively.
Observe that
\begin{eqnarray}\label{intertwining relation}
\Phi\circ \left(\sum a_0\prod_{i=1}^n[D_1,a_i]\right)=\left(\sum \phi(a_0)\prod_{i=1}^n[D_2,\phi(a_i)]\right)\circ \Phi\,.
\end{eqnarray}
Consider an arbitrary element $\,\xi\in\pi(d_1J_0^n)$. By definition, $\,\xi=\sum \prod_{i=0}^n[D_1,a_i]\in \pi(d_1J_0^n)$ such that $\sum a_0
\prod_{i=1}^n[D_1,a_i]=0$. Now, using equation (\ref{intertwining relation}) and $\Phi$ is a unitarity (surjectivity is enough), we have
\begin{center}
$\sum \phi(a_0)\prod_{i=1}^n[D_2,\phi(a_i)]=0.$
\end{center}
This shows well-definedness of $\Psi$. Now it is easy to check that $\Psi$ is a dga morphism.
\end{proof}

\begin{remark}\rm
One can weaken the definition of morphism of spectral triples by demanding the map $\Phi$ to be only linear. This was defined in (\cite{bc}). But
for Proposition (\ref{Connes is functor}) to hold one requires surjectivity of $\Phi$. However, the reason why we have assumed $\Phi$ to be unitary
will be justified in the next section. 
\end{remark}

\begin{lemma}
The quantum double suspension of a spectral triple is a covariant functor $\varSigma^2$ on the category $\,\mathcal{S}pec$.
\end{lemma}
\begin{proof}
Easy to verify.
\end{proof}

\begin{proposition}
The functor $\varSigma^2$ gives an equivalence $\varSigma^2(\mathcal{S}pec)\cong\mathcal{S}pec\,$ of categories, and hence $\varSigma^2$ is not a
constant functor.
\end{proposition}
\begin{proof}
Recall that as a linear space $\varSigma^2\mathcal{A}=\mathcal{A}\otimes\mathcal{S}\bigoplus\mathbb{C}[z,z^{-1}]$, and $\varSigma^2D=D\otimes 1+
F\otimes N$. Suppose $(\phi,\Phi):(\varSigma^2\mathcal{A}_1,\varSigma^2\mathcal{H}_1,\varSigma^2D_1)\rightarrow(\varSigma^2\mathcal{A}_2,
\varSigma^2\mathcal{H}_2,\varSigma^2D_2)$ is an isomorphism in the sense of Definition (\ref{category}). One can replace $N$ by $N+g(N)$ for a
suitable function $g$ such that $(\varSigma^2\mathcal{A},\varSigma^2\mathcal{H},D\otimes I+F\otimes(N+g(N)))$ remains an honest spectral triple,
and computations done in (\cite{cg}) does not get affected. It is possible to choose such a function $g$ so that the following map
\begin{align*}
\sigma(|D_1|)\times\sigma(N+g(N)) &\longrightarrow \mathbb{N}_+\\
(\lambda_n,n+g(n)) &\longmapsto \lambda_n+n+g(n)
\end{align*}
becomes one to one. This is possible since $D_1$ has discrete spectrum. This will imply that any unitary $\widetilde{\Phi}:\mathcal{H}_1\otimes
\ell^2(\mathbb{N})\rightarrow\mathcal{H}_2\otimes\ell^2(\mathbb{N})$ is of the form $\Phi\otimes 1$, where $\Phi:\mathcal{H}_1\rightarrow
\mathcal{H}_2$ is a unitary. This will asure that algebra isomorphism $\widetilde{\phi}:\varSigma^2\mathcal{A}_1\rightarrow\varSigma^2\mathcal{A}_2$
is of the form $\phi\otimes 1\bigoplus 1$, where $\phi:\mathcal{A}_1\rightarrow\mathcal{A}_2$ is an algebra isomorphism. This shows that
\begin{center}
$(\varSigma^2\mathcal{A}_1,\varSigma^2\mathcal{H}_1,\varSigma^2D_1)\cong(\varSigma^2\mathcal{A}_2,\varSigma^2\mathcal{H}_2,\varSigma^2D_2)
\Longrightarrow(\mathcal{A}_1,\mathcal{H}_1,D_1)\cong(\mathcal{A}_2,\mathcal{H}_2,D_2)$.
\end{center}
The other implication `$\Leftarrow$' is obvious.
\end{proof}

Recall Theorem ($3.22$) from (\cite{cg}).
\begin{theorem}[\cite{cg}]\label{final thm}
For the spectral triple $\,\left(\varSigma^2\mathcal{A},\varSigma^2\mathcal{H},\varSigma^2D\right)$, we have
\begin{enumerate}
\item $\Omega_{\varSigma^2 D}^1\left(\varSigma^2 \mathcal{A}\right)\cong\Omega_D^1(\mathcal{A})\otimes\mathcal{S}\bigoplus\varSigma^2\mathcal{A}\,$.
\item $\Omega_{\varSigma^2 D}^n\left(\varSigma^2 \mathcal{A}\right)\cong\Omega_D^n(\mathcal{A})\otimes\mathcal{S}\,$, $\,$ for all $n\geq 2\,$.
\item The differential $\,\,\delta^0:\varSigma^2 \mathcal{A}\longrightarrow\Omega_{\varSigma^2 D}^1\left(\varSigma^2\mathcal{A}\right)\,$ is given by,
\begin{center}
$a\otimes T+f\longmapsto [D,a]\otimes T\bigoplus\left(a\otimes[N,T]+f^\prime\right)$.
\end{center}
\item The differential $\,\,\delta^1:\Omega_{\varSigma^2 D}^1\left(\varSigma^2\mathcal{A}\right)\longrightarrow\Omega_{\varSigma^2 D}^2
\left(\varSigma^2\mathcal{A}\right)\,$ is given by,
\begin{center}
$\delta^1|_{\Omega_D^1(\mathcal{A})\otimes\mathcal{S}}=d^1\otimes 1\,\,$ and $\,\,\delta^1|_{\varSigma^2 \mathcal{A}}=0$.
\end{center}
\item The differential $\,\,\delta^n:\Omega_{\varSigma^2 D}^n\left(\varSigma^2 \mathcal{A}\right)\longrightarrow\Omega_{\varSigma^2 D}^{n+1}
\left(\varSigma^2\mathcal{A}\right)\,$ is given by $\,\delta^n=d^n\otimes 1$ for all $n\geq 2\,$.
\end{enumerate}
Here, $\,d:\Omega_D^\bullet(\mathcal{A})\longrightarrow\Omega_D^{\bullet+1}(\mathcal{A})$ is the differential of the Dirac dga.
\end{theorem}

\begin{remark}\rm
The dga $\,\Omega_{\varSigma^2 D}^\bullet\left(\varSigma^2 \mathcal{A}\right)$ can be described alternatively as follows. Notice that for the
$($graded$)$ algebra $\Omega_D^\bullet(\mathcal{A})$ one can consider $\varSigma^2\left(\Omega_D^\bullet(\mathcal{A})\right)=\Omega_D^\bullet
(\mathcal{A})\otimes\mathcal{S}\bigoplus\mathbb{C}[z,z^{-1}]$. This is a graded algebra whose degree zero term is $\,\mathcal{A}\otimes\mathcal{S}
\bigoplus\mathbb{C}[z,z^{-1}]=\varSigma^2\mathcal{A}$, and degree $n$ term is $\,\Omega_D^n(\mathcal{A})\otimes\mathcal{S}$ for $n\geq 1$.
That is,
\begin{center}
$\varSigma^2\left(\Omega_D^\bullet(\mathcal{A})\right)=\varSigma^2\mathcal{A}\bigoplus\Omega_D^1(\mathcal{A})\otimes\mathcal{S}\bigoplus
\Omega_D^2(\mathcal{A})\otimes\mathcal{S}\bigoplus\ldots\ldots$
\end{center}
as a graded algebra. Then, as a graded algebra $\,\Omega_{\varSigma^2 D}^\bullet\left(\varSigma^2 \mathcal{A}\right)=\varSigma^2\left(\Omega_D^
\bullet(\mathcal{A})\right)\bigoplus\varSigma^2\mathcal{A}$, where $\varSigma^2\mathcal{A}$ sits in the degree $1$ term.
\end{remark}

\begin{corollary}
The Cohomology of $\left(\,\Omega_{\varSigma^2 D}^\bullet(\varSigma^2\mathcal{A}),\delta^\bullet\right)$ is given by
\begin{enumerate}
\item $H^0(\varSigma^2\mathcal{A})=H^0(\mathcal{A})\otimes \mathcal{S}_{diag}\bigoplus\mathbb{C}$.
\item $H^1(\varSigma^2\mathcal{A})=H^1(\mathcal{A})\otimes \mathcal{S}_{diag}\bigoplus\mathcal{A}\otimes \mathcal{S}_{diag}\bigoplus Ker(d^1)\otimes
\mathcal{S}_{off}\bigoplus\mathbb{C}$.
\item $H^n(\varSigma^2\mathcal{A})=H^n(\mathcal{A})\otimes\mathcal{S}$, $\,$ for all $n\geq 2\, $.
\end{enumerate}
where $H^\bullet(\mathcal{A})$ denotes the cohomology of $(\,\Omega_D^\bullet(\mathcal{A}),d^\bullet)$, and $\,\mathcal{S}_{diag}\,,
\mathcal{S}_{off}$ denote spaces of finitely supported diagonal and off-diagonal matrices respectively.
\end{corollary}
\begin{proof}
We have $H^0(\varSigma^2\mathcal{A})=Ker(\delta^0)$. Writing $a\otimes T=\sum_{i,j}a_{ij}\otimes e_{ij}$ in terms of elementary matrices $(e_{ij})$
we have,
\begin{eqnarray*}
Ker(\delta^0) & = & \{a_{ij}=0\,\,\,for\,i\neq j\,,\,[D,a_{ii}]=0\,\,\forall\,i\,,\,f=\,constant\}\\
& = & Ker(d^0)\otimes\mathcal{S}_{diag}\bigoplus\mathbb{C}\\
& = & H^0(\mathcal{A})\otimes\mathcal{S}_{diag}\bigoplus\mathbb{C}\,.
\end{eqnarray*}
This proves part $(1)$. For part $(2)$ observe that
\begin{center}
$Ker(\delta^1)=\,Ker(d^1)\otimes \mathcal{S}\bigoplus\varSigma^2\mathcal{A}$.
\end{center}
and $Im(\delta^0)=Im(\delta^0|_{\mathcal{A}\otimes\mathcal{S}})\bigoplus\mathbb{C}[z,z^{-1}]/\mathbb{C}$. Hence, we need to determine
$\frac{Ker(d^1)\otimes\mathcal{S}\bigoplus\mathcal{A}\otimes\mathcal{S}}{Im(\delta^0|_{\mathcal{A}\otimes\mathcal{S}})}$. Now,
\begin{eqnarray*}
\delta^0 :\,\mathcal{A}\otimes\mathcal{S}_{off}\bigoplus\mathcal{A}\otimes\mathcal{S}_{diag} & \longrightarrow & \left(Ker(d^1)\oplus\mathcal{A}
\right)\otimes\mathcal{S}_{off}\bigoplus\left(Ker(d^1)\oplus\mathcal{A}\right)\otimes\mathcal{S}_{diag}\\
\left(\sum_{i\neq j}a_{ij}\otimes e_{ij}\,,\sum_ib_i\otimes e_{ii}\right) & \longmapsto & \left(\sum_{i\neq j}\left(d^0a_{ij},(i-j)a_{ij}\right)
\otimes e_{ij}\,,\sum_i\left(d^0b_i,0\right)\otimes e_{ii}\right)
\end{eqnarray*}
Hence, $\delta^0=\delta^0_1\oplus\delta^0_2$. Observe that $Im(\delta^0_2)=Im(d^0)\otimes\mathcal{S}_{diag}$. Now,
\begin{eqnarray*}
\Psi:\,\frac{\left(Ker(d^1)\oplus\mathcal{A}\right)\otimes\mathcal{S}_{off}}{Im(\delta^0_1)} & \longrightarrow & Ker(d^1)\otimes\mathcal{S}_{off}\\
\sum_{i\neq j}\left[(\omega_{ij},a_{ij})\otimes e_{ij}\right] & \longmapsto & \sum_{i\neq j}\left(\omega_{ij}-(i-j)^{-1}d^0a_{ij}\right)\otimes e_{ij}
\end{eqnarray*}
is a well-defined linear isomorphism. Hence,
\begin{eqnarray*}
H^1(\varSigma^2\mathcal{A}) & = & \frac{\left(Ker(d^1)\oplus\mathcal{A}\right)\otimes\mathcal{S}_{off}}{Im(\delta^0_1)}\bigoplus\frac{Ker(d^1)}
{Im(d^0)}\otimes\mathcal{S}_{diag}\bigoplus\mathcal{A}\otimes\mathcal{S}_{diag}\bigoplus\mathbb{C}\\
& = & Ker(d^1)\otimes\mathcal{S}_{off}\bigoplus H^1(\mathcal{A})\otimes\mathcal{S}_{diag}\bigoplus\mathcal{A}\otimes\mathcal{S}_{diag}\bigoplus\mathbb{C}.
\end{eqnarray*}
This proves part $(2)$, and part $(3)$ is easy to verify.
\end{proof}

If $\mathcal{A}$ comes with a decreasing filtration $$\mathcal{A}=\mathcal{A}_0\supseteq\mathcal{A}_{-1}\supseteq\ldots\ldots\supseteq\{0\}$$ then
the algebra $\varSigma^2\mathcal{A}$ has the induced filtration. By Lemma (\ref{filtered algebra}), $\,\Omega_{\varSigma^2 D}^\bullet(\varSigma^2
\mathcal{A})$ then becomes a filtered algebra.

\begin{proposition}
The associated graded algebra of the filtered algebra $\,\Omega_{\varSigma^2 D}^\bullet(\varSigma^2\mathcal{A})$ is $$\mathcal{G}(\varSigma^2
\mathcal{A})=\bigoplus_{n\leq 0}\left(\frac{\mathcal{A}_n}{\mathcal{A}_{n-1}}\bigoplus\left(\frac{\Omega^1(\mathcal{A}_n)}{\Omega^1(\mathcal{A}_{n-1})}
\oplus\frac{\mathcal{A}_n}{\mathcal{A}_{n-1}}\right)\bigoplus\left(\bigoplus_{p\geq 2}\frac{\Omega^p(\mathcal{A}_n)}{\Omega^p(\mathcal{A}_{n-1})+
dJ_0^{p-1}(\mathcal{A}_n)}\right)\right)\otimes\mathcal{S}\,.$$
Hence, $\mathcal{G}(\varSigma^2\mathcal{A})$ depends only on the filtration of $\mathcal{A}$.
\end{proposition}
\begin{proof}
By Lemma (\ref{associated graded algebra}), the associated graded algebra is $$\mathcal{G}(\varSigma^2\mathcal{A})=\bigoplus_{n\leq 0}
\bigoplus_{p\geq 0}\frac{\Omega^p(\varSigma^2\mathcal{A}_n)}{\Omega^p(\varSigma^2\mathcal{A}_{n-1})+J^{p,n}(\varSigma^2\mathcal{A})}\,.$$
For $p=0$,
\begin{eqnarray*}
\frac{\Omega^p(\varSigma^2\mathcal{A}_n)}{\Omega^p(\varSigma^2\mathcal{A}_{n-1})+J^{p,n}(\varSigma^2\mathcal{A})} & \cong & \frac{\varSigma^2
\mathcal{A}_n}{\varSigma^2\mathcal{A}_{n-1}}\\
& \cong & \frac{\mathcal{A}_n}{\mathcal{A}_{n-1}}\otimes\mathcal{S}\,,
\end{eqnarray*}
and for $p\geq 2$,
\begin{eqnarray*}
&   & \frac{\Omega^p(\varSigma^2\mathcal{A}_n)}{\Omega^p(\varSigma^2\mathcal{A}_{n-1})+J^{p,n}(\varSigma^2\mathcal{A})}\\
& \cong & \frac{\pi(\Omega^p(\varSigma^2\mathcal{A}_n))}{\pi(\Omega^p(\varSigma^2\mathcal{A}_{n-1}))+\pi(dJ_0^{p-1,n}(\varSigma^2\mathcal{A}))}\\
& \cong & \frac{\pi(\Omega^p(\mathcal{A}_n\otimes\mathcal{S}))\bigoplus\pi(\Omega^p(\mathbb{C}[z,z^{-1}]))}{\left(\pi(\Omega^p(\mathcal{A}_{n-1}\otimes
\mathcal{S}))+\pi(dJ_0^{p-1}(\mathcal{A}_n\otimes\mathcal{S}))\right)\bigoplus\pi(\Omega^p(\mathbb{C}[z,z^{-1}]))}\\
& \cong & \frac{\pi(\Omega^p(\mathcal{A}_n\otimes\mathcal{S}))}{\pi\left(\Omega^p(\mathcal{A}_{n-1}\otimes\mathcal{S})+dJ_0^{p-1}(\mathcal{A}_n
\otimes\mathcal{S})\right)}\\
& \cong & \frac{\pi(\Omega^p(\mathcal{A}_n))\otimes\mathcal{S}}{\pi\left(\Omega^p(\mathcal{A}_{n-1})+dJ_0^{p-1}(\mathcal{A}_n)\right)\otimes\mathcal{S}}\\
& \cong & \frac{\Omega^p(\mathcal{A}_n)}{\Omega^p(\mathcal{A}_{n-1})+dJ_0^{p-1}(\mathcal{A}_n)}\otimes\mathcal{S}
\end{eqnarray*}
by Proposition $(3.8)$ and $(3.10)$ in (\cite{cg}). Finally, for $p=1$
\begin{eqnarray*}
\frac{\Omega^1(\varSigma^2\mathcal{A}_n)}{\Omega^1(\varSigma^2\mathcal{A}_{n-1})+J^{1,n}(\varSigma^2\mathcal{A})} & \cong &
\frac{\Omega^1(\mathcal{A}_n)}{\Omega^1(\mathcal{A}_{n-1})}\otimes\mathcal{S}\bigoplus\frac{\mathcal{A}_n}{\mathcal{A}_{n-1}}\otimes\mathcal{S}\,.
\end{eqnarray*} by part $(1)$ of Th. $3.20$ in (\cite{cg}). Hence, our claim follows.
\end{proof}
\bigskip

%%%%%%%%%%%%%%%%%%%%%%%%%%%%%%%  Computation for FGR dga  %%%%%%%%%%%%%%%%%%%%%%%%%%%%%%%%%%%%%%%%%%%%%%%%%%%%%%%%%%%%%%

\section{FGR DGA and The Quantum Double Suspension}

In this section our objective is to compute the dga of Fr\"{o}hlich et al. for the quantum double suspension. We first recall its definition from
(\cite{fgr}).
\begin{definition}
For any $p$-summable spectral triple $(\mathcal{A},\mathcal{H},D)$ consider the following functional
\begin{eqnarray}\label{Frohlich defn.}
\begin{aligned}
\int:\pi(\Omega^\bullet(\mathcal{A})) &\longrightarrow \mathbb{C}\\
[v] &\longmapsto \lim_{\varepsilon\rightarrow 0^+}\,\frac{Tr_{\mathcal{H}}(v e^{-\varepsilon D^2})}{Tr_{\mathcal{H}}(e^{-\varepsilon D^2})}
\end{aligned}
\end{eqnarray}
Let $$K(\mathcal{A}):=\bigoplus_{n=0}^\infty K^n(\mathcal{A})\,,\quad K^n(\mathcal{A}):=\left\{\omega\in\Omega^n(\mathcal{A})\,\colon\int\pi
(\omega)^*\pi(\omega)=0\right\}.$$
Then $$\widetilde \Omega_D^\bullet(\mathcal{A}):=\bigoplus_{n=0}^\infty\widetilde \Omega_D^n(\mathcal{A})\,,\quad\widetilde \Omega_D^n(\mathcal{A})
:=\Omega^n(\mathcal{A})/(K^n+dK^{n-1})\cong\pi(\Omega^n(\mathcal{A}))/\pi(K^n+dK^{n-1})$$ is a differential graded algebra called the FGR dga.
\end{definition}

\begin{remark}\rm
\begin{enumerate}
\item Note that for a $p$-summable spectral triple $(\mathcal{A},\mathcal{H},D)$,
\begin{center}
 $Lim_{\lambda\rightarrow \infty}\left(\frac{1}{\lambda}Tr\left(Te^{-\lambda^{-2/p}D^2}\right)\right)=\Gamma(\frac{p}{2}+1)Tr_\omega(T|D|^{-p})$
\end{center}
for all $\,T\in\mathcal{B}(\mathcal{H})\,($\cite{con}, Page $563)$. Hence, the funcional considered in equation $(\ref{Frohlich defn.})$ is nothing
but the Dixmier trace $Tr_\omega$ upto a positive constant.
\item Since, for any compact operator $K\in\mathcal{K}(\mathcal{H}),\,Tr_\omega(K|D|^{-p})=0$ the functional in (\ref{Frohlich defn.}) is
well-defined on $\,\pi\left(\Omega^\bullet(\mathcal{A})\right)\subseteq\mathcal{B}(\mathcal{H})/\mathcal{K}(\mathcal{H})$.
\item For the classical case of manifolds and the noncommutative torus, $K^n=J_0^n$ (Def. \ref{our defn of Connes}). Hence, the FGR dga coincides
with the Dirac dga in these cases (\cite{fgr}).
\end{enumerate}
\end{remark}

\begin{lemma}
The association $\mathcal{G}:(\mathcal{A},\mathcal{H},D)\longmapsto\widetilde \Omega_D^\bullet(\mathcal{A})$ gives a covariant functor from
$\mathcal{S}pec$ to $DGA$, the category of differential graded algebras over $\mathbb{C}$.
\end{lemma}
\begin{proof}
Consider two objects $(\mathcal{A}_1,\mathcal{H}_1,D_1),(\mathcal{A}_2,\mathcal{H}_2,D_2)\in \mathcal{O}b(\mathcal{S}pec)$ and suppose there is a
morphism $(\phi\,,\Phi):(\mathcal{A}_1,\mathcal{H}_1,D_1)\longrightarrow (\mathcal{A}_2,\mathcal{H}_2,D_2)$. Define
\begin{align*}
\Psi:\widetilde \Omega_{D_1}^\bullet(\mathcal{A}_1) &\longrightarrow \widetilde \Omega_{D_2}^\bullet(\mathcal{A}_2)\\
\left[\sum a_0\prod_{i=1}^n[D_1,a_i]\right] &\longmapsto\left[\sum \phi(a_0)\prod_{i=1}^n[D_2,\phi(a_i)]\right]
\end{align*}
for all $\,a_j\in \mathcal{A}_1,\,n\geq 0\,$. To show $\Psi$ is well-defined we must show that $\Psi(\pi_1(K_1^m))\subseteq \pi_2(K_2^m)$
for all $m\geq 0$. Observe that
\begin{eqnarray}
\Phi\circ\left(\sum a_0\prod_{i=1}^n[D_1,a_i]\right)=\left(\sum \phi(a_0)\prod_{i=1}^n[D_2,\phi(a_i)]\right)\circ\Phi\,.
\end{eqnarray}
Now, $\,\Phi D_1=D_2\Phi$ will imply that $\,\Phi e^{-tD_1^2}\Phi^*=e^{-tD_2^2}$. Let us denote
\begin{align*}
\pi_1(\omega) &:= \sum a_0\prod_{i=1}^n[D_1,a_i]\\
\pi_2(\widetilde \omega) &:= \sum \phi(a_0)\prod_{i=1}^n[D_2,\phi(a_i)]
\end{align*}
Now,
\begin{eqnarray*}
Tr(\pi_2(\widetilde \omega)^*\pi_2(\widetilde \omega)e^{-tD_2^2})
& = & Tr(\pi_2(\widetilde \omega)^*\pi_2(\widetilde \omega)\Phi e^{-tD_1^2}\Phi^*)\\
& = & Tr(\pi_1(\omega)^*\Phi^*\Phi\pi_1(\omega)e^{-tD_1^2})\\
& = & Tr(\pi_1(\omega)^*\pi_1(\omega)e^{-tD_1^2})
\end{eqnarray*}
and $Tr(e^{-tD_1^2})=Tr(e^{-tD_2^2})$. This proves that $\Psi(\pi_1(K_1^m))=\pi_2(K_2^m)$ i,e. $\Psi$ is well-defined, and one can check that it
is a dga morphism.
\end{proof}

\begin{remark}\rm
\begin{enumerate}
\item Although, surjectivity of $\,\Phi$ is enough to ensure that Dirac dga is a functor, it fails in this case of FGR dga. This is the reason
we have chosen $\Phi$ to be unitary. Unless $\Phi$ is both one-one and onto it is not guaranteed that $\Psi(\pi_1(K_1^m))\subseteq\pi_2(K_2^m)$.
\item One may come up with a different definition of the category $\,\mathcal{S}pec$ of spectral triples which allows larger set of morphisms than
ours; such that both the Dirac dga and FGR dga become functor. Here we stress to the point that \textbf{it will not condradict} our main result in
this article as we shall see now. Because of this reason we have chosen the simplest possible definition for the category $\mathcal{S}pec$.
\end{enumerate}
\end{remark}

To make the computation possible we need to use the functional in (\ref{Frohlich defn.}) in a different disguise, namely
\begin{eqnarray}\label{our defn. of Frohlich}
\begin{aligned}
\oint:\pi\left(\Omega^\bullet(\mathcal{A})\right) &\longrightarrow \mathbb{C}\\
[\widetilde{v}] &\longmapsto\,\lim_{t\rightarrow 0}\left(t^p Tr\left(\widetilde v e^{-t|D|}\right)\right)
\end{aligned}
\end{eqnarray}

Well-definedness of this functional follows from the next lemma.
\begin{lemma}\label{justification for our defn}
Let $(\mathcal{A},\mathcal{H},D)$ be a $p$-summable spectral triple. Then the functional $\oint$ is equal to the Dixmier trace upto a positive
constant $($which depends only on $p)$. 
\end{lemma}
\begin{proof}
Recall the following equality
\begin{center}
$\omega\left(\frac{1}{t}Tr\left(exp(-(tA)^{-q})B\right)\right)=\Gamma\left(1+\frac{1}{q}\right)Tr_\omega(AB)$
\end{center}
proved in (\cite{suz}) for any $B\in\mathcal{B}(\mathcal{H})$. Now take $q=1/p$ and $A=|D|^{-p}$.
\end{proof}

\begin{corollary}
For any $\,T_1\otimes T_2\in\mathcal{B}(\mathcal{H}\otimes\ell^2(\mathbb{N}))\,$,
\begin{center}
$t^{p+1}Tr\left((T_1 \otimes T_2) e^{-t|\varSigma^2 D|}\right)=t^pTr\left(T_1 e^{-t|D|}\right)\, tTr\left(T_2 e^{-tN}\right)\,$.
\end{center}
\end{corollary}

\begin{remark}\rm
It is this Corollary which makes the computation in this section possible. Moreover, since both the funcionals $\int$ and $\oint$ become equal upto
a constant, and we are interested in the spaces $K^n$ in Definition $(\ref{Frohlich defn.})$, it is absolutely permissible to choose $\oint$ over $\int$.
\end{remark}

\begin{lemma}\label{n=0 part}
$K^0(\varSigma^2\mathcal{A})=\mathcal{A}\otimes\mathcal{S}$.
\end{lemma}

\begin{proof}
Choose any element $\sum_ka_k\otimes T_k$ of $\mathcal{A}\otimes\mathcal{S}$. In terms of elementary matrices we can write $\,T_k=\sum_{i,j}
\,\alpha_{ij}^{(k)}e_{ij}\,$.
Then
\begin{eqnarray*}
\oint (\sum_k a_k\otimes T_k)(\sum_k a_k\otimes T_k)^*  & = & \oint (\sum_{k,i,j} a_{kij}\otimes e_{ij})(\sum_{k,i,j} a_{kij}^*\otimes e_{ji})\\
& = & \oint \sum_{k,k^\prime ,i,j,i^\prime} a_{kij}a_{k^\prime i^\prime j}^*\otimes e_{ii^\prime}\\
& = & \, \lim_{t\rightarrow 0}\,t^{p+1}Tr\left(\left(\sum_{k,k^\prime,i,j,i^\prime} a_{kij}a_{k^\prime i^\prime j}^* \otimes e_{ii^\prime}\right)
e^{-t|\varSigma^2 D|} \right)\\
& = & \sum_{k,k^\prime,i,j,i^\prime} \lim_{t\rightarrow 0}\left(t^pTr\left(a_{kij}a_{k^\prime i^\prime j}^* e^{-t|D|}\right)\right)\left(tTr
\left(e_{ii^\prime}e^{-tN}\right)\right)\\
& = & \sum_{k,k^\prime,i,j} \lim_{t\rightarrow 0} \left(t^pTr\left(a_{kij}a_{k^\prime ij}^* e^{-t|D|}\right)\right)(te^{-ti})\\
& = & 0\,.
\end{eqnarray*}
Hence, $\,\mathcal{A}\otimes\mathcal{S}\subseteq K^0(\varSigma^2\mathcal{A})$. Now, for arbitrary $\,\sum_ka_k\otimes T_k+f\in K^0(\varSigma^2\mathcal{A})$,
\begin{eqnarray*}
0 & = & \oint (\sum_k a_k\otimes T_k +f)(\sum_k a_k\otimes T_k +f)^*\\
& = & \oint (\sum_k a_k\otimes T_k)(\sum_k a_k\otimes T_k)^* + \oint ff^* + \oint f(\sum_k a_k\otimes T_k)^*+ \oint (\sum_k a_k\otimes T_k)f^*\\
& = & \oint ff^*
\end{eqnarray*}
because same calculation as above proves that both $\oint f\left(\sum_k a_k\otimes T_k\right)^*$ and $\oint \left(\sum_k a_k\otimes T_k\right)f^*$
are zero. For any $f\in\mathbb{C}[z,z^{-1}]$, $\oint ff^*$ is just the integration of the function $\,ff^*\equiv |f|^2\,$ against the Haar
measure on $S^1$. This shows that $f=0$ i,e. $K^0(\varSigma^2 \mathcal{A})\subseteq\mathcal{A}\otimes\mathcal{S}\,$.
\end{proof}

\begin{remark}\rm
In (Assumption $2.13$, Page $131$ in \cite{fgr}), authors have assumed that $K^0=\{0\}$. Previous Lemma $(\ref{n=0 part})$ shows that this is
never true in the case of quantum double suspension.
\end{remark}

\begin{lemma}\label{to be used in next lemma}
$\oint(F\otimes 1)\pi(\omega)=0$ for any $\,\omega\in\Omega^1(\mathcal{A}\otimes\mathcal{S})$.
\end{lemma}

\begin{proof}
Let
\begin{eqnarray*}
\pi(\omega) & = & \sum (a_0\otimes T_0)[\varSigma^2D,a_1\otimes T_1]\\
            & = & \sum a_0[D,a_1]\otimes T_0T_1+Fa_0a_1\otimes T_0[N,T_1].
\end{eqnarray*}
Then, using elementary matrices $(e_{ij})$ we have
\begin{eqnarray*}
\oint (F\otimes 1)\pi(\omega) & = & \,Lim_{t\rightarrow 0}\left(t^{p+1}Tr\left(\pi(\omega)e^{-t|\varSigma^2D|}\right)\right)\\
& = & \sum \,\lim_{t\rightarrow 0}\left(t^{p+1}Tr\left((a_0[D,a_1]\otimes T_0T_1)e^{-t|\varSigma^2D|}\right)\right)\\
&  & \quad\,+\,\,\lim_{t\rightarrow 0}\left(t^{p+1}Tr\left((Fa_0a_1\otimes T_0[N,T_1])e^{-t|\varSigma^2D|}\right)\right)\\
& = & \sum \,\lim_{t\rightarrow 0}\left(t^{p+1}Tr\left(\sum_{i,j,q} \left(a_{0ij}[D,a_{1jq}]\otimes e_{iq}\right)e^{-t|\varSigma^2D|}\right)\right)\\
&  & \quad\,+\,\,\lim_{t\rightarrow 0}\left(t^{p+1}Tr\left(\left(\sum_{i,j,q}Fa_{0ij}a_{1jq}(j-q)\otimes e_{iq}\right)e^{-t|\varSigma^2D|}\right)\right)\\
& = & \sum\sum_{i,j,q}\,\lim_{t\rightarrow 0}\left(t^pTr\left(a_{0ij}[D,a_{1jq}]e^{-t|D|}\right)tTr\left(e_{iq}e^{-tN}\right)\right)\\
&  & \quad\quad\quad\,+\,\,\lim_{t\rightarrow 0}\left(t^pTr\left(Fa_{0ij}a_{1jq}(j-q)e^{-t|D|}\right)tTr\left(e_{iq}e^{-tN}\right)\right)\\
& = & \sum\sum_{i,j}\,\lim_{t\rightarrow 0}\left(t^pTr\left(a_{0ij}[D,a_{1ji}]e^{-t|D|}\right)\left(te^{-ti}\right)\right)\\
&  & \quad\quad\quad\,+\,\,\lim_{t\rightarrow 0}\left(t^pTr\left(Fa_{0ij}a_{1ji}(j-i)e^{-t|D|}\right)\left(te^{-ti}\right)\right)\\
& = & 0
\end{eqnarray*}
and this concludes the proof.
\end{proof}

\begin{lemma}\label{Main calculation}
$\pi\left(K^1(\varSigma^2 \mathcal{A})\right) = \pi\left(\Omega^1(\mathcal{A}\otimes \mathcal{S}) \right) \bigoplus \pi\left(K^1(\mathbb{C}
[z,z^{-1}])\right)\, $.
\end{lemma}

\begin{proof}
We first prove that $\pi\left(\Omega^1(\mathcal{A}\otimes\mathcal{S}\right)\subseteq\pi\left(K^1(\varSigma^2\mathcal{A})\right)$. Arbitrary
element of $\pi\left(\Omega^1(\mathcal{A}\otimes\mathcal{S})\right)$ looks like $\,\pi(\omega)=\sum_k(a_{0k}\otimes T_{0k})[\varSigma^2D,a_{1k}
\otimes T_{1k}]$. Then, using elementary matrices $(e_{ij})$ we get
\begin{eqnarray*}
\pi(\omega) & = & \sum_k(\sum_{i,j}a_{0kij}\otimes e_{ij})[\varSigma^2 D,\sum_{p,q}a_{1kpq}\otimes e_{pq}]\\
            & = & \sum_{k,i,j,q}\left(a_{0kij}[D,a_{1kjq}]+Fa_{0kij}a_{1kjq}(j-q)\right)\otimes e_{iq}\,,
\end{eqnarray*}
and
\begin{center}
$\pi(\omega)^* \,\,=\,\, \sum_{k,i,j,q} \left(a_{0kij}[D,a_{1kjq}]+Fa_{0kij}a_{1kjq}(j-q)\right)^*\otimes e_{qi}\,.$
\end{center}
Let $\,T_{kjiq}=a_{0kij}[D,a_{1kjq}]+Fa_{0kij}a_{1kjq}(j-q)$. Now,
\begin{eqnarray*}
\oint \pi(\omega)\pi(\omega)^* & = & \, \lim_{t \rightarrow 0}\,t^{p+1}Tr\left(\pi(\omega)\pi(\omega)^*e^{-t|\varSigma^2 D|}\right)\\
& = & \, \lim_{t \rightarrow 0}\,t^{p+1}Tr\left(\left(\sum_{i,q}\sum_{k,j} T_{kjiq}\otimes e_{iq}\right)\left(\sum_{i^\prime,q^\prime}\sum_{k,j}
T_{kji^\prime q^\prime}^*\otimes e_{q^\prime i^\prime}\right)e^{-t|\varSigma^2 D|} \right)\\
& = & \, \lim_{t \rightarrow 0}\,t^{p+1}Tr\left(\left(\sum_{i,q,i^\prime}\left(\sum_{k,j} T_{kjiq}\right)\left(\sum_{k,j} T_{kji^\prime q}^*\right)
\otimes e_{ii^\prime}\right)e^{-t|\varSigma^2 D|}\right)\\
& = & \, \lim_{t \rightarrow 0} \sum_{i,q,i^\prime} t^pTr\left(\left(\sum_{k,j} T_{kjiq}\right)\left(\sum_{k,j} T_{kji^\prime q}\right)^*e^{-t|D|}
\right)tTr\left(e_{ii^\prime} e^{-tN}\right)\\
& = & \, \lim_{t \rightarrow 0} \sum_{i,q} t^pTr\left(\left(\sum_{k,j} T_{kjiq}\right)\left(\sum_{k,j} T_{kjiq}\right)^*e^{-t|D|} \right)
\left(te^{-ti}\right)\\
& = & 0\,.
\end{eqnarray*}
Hence, $\,\pi\left(\Omega^1(\mathcal{A}\otimes\mathcal{S})\right)\bigoplus\pi\left(K^1(\mathbb{C}[z,z^{-1}])\right)\subseteq\pi\left(K^1(\varSigma^2
\mathcal{A})\right)$.\\
To show the converse choose $\,\pi(\omega)=\sum_k(a_{0k}\otimes T_{0k}+f_{0k})[\varSigma^2D,a_{1k}\otimes T_{1k}+f_{1k}]\,$, an element in
$\pi\left(K^1(\varSigma^2\mathcal{A})\right)$. Then,
\begin{center}
$\pi(\omega)\,\,=\,\,\sum_k F\otimes f_{0k}f_{1k}^\prime+\pi(\widetilde \omega)$
\end{center}
where, $\,\pi(\widetilde \omega)\in\pi\left(\Omega^1(\mathcal{A}\otimes\mathcal{S})\right)$. Hence, $\,\pi(\omega)^*=\pi(\widetilde \omega)^*+
\sum_k F\otimes(f_{0k}f_{1k}^\prime)^*$. Since, $\,\pi(\omega)\in\pi\left(K^1(\varSigma^2 \mathcal{A})\right)$ we have
\begin{eqnarray*}
0 & = & \oint \pi(\omega)^*\, \pi(\omega)\\
& = & \oint\pi(\widetilde\omega)^*\,\pi(\widetilde \omega)+\oint\left(\sum_kF\otimes(f_{0k}f_{1k}^\prime)^*\right)\pi(\widetilde \omega)+
\oint\pi(\widetilde \omega)^*\left(\sum_kF\otimes(f_{0k}f_{1k}^\prime)\right)\\
&   & + \oint\left(\sum_kf_{0k}f_{1k}^\prime\right)^*\left(\sum_kf_{0k}f_{1k}^\prime\right)\,.
\end{eqnarray*}
This shows that $\,\oint(\sum_k f_{0k}f_{1k}^\prime)^*(\sum_k f_{0k}f_{1k}^\prime) = 0\,$ (using Lemma \ref{to be used in next lemma}). That is,
$\,\sum_kF\otimes f_{0k}f_{1k}^\prime \in \pi\left(K^1(\mathbb{C}[z,z^{-1}])\right)$. Hence, $\,\pi\left(K^1(\varSigma^2\mathcal{A})\right)
\subseteq\pi\left(\Omega^1(\mathcal{A}\otimes\mathcal{S})\right)\bigoplus\pi\left(K^1(\mathbb{C}[z,z^{-1}])\right)\,$.
\end{proof}

\begin{proposition}\label{to refer in thm.}
$\widetilde \Omega_{\varSigma^2 D}^1(\varSigma^2 \mathcal{A})\cong\mathbb{C}[z,z^{-1}]\,$ as $\varSigma^2\mathcal{A}\,$-bimodule. 
\end{proposition}

\begin{proof}
We have $\,\widetilde \Omega_{\varSigma^2 D}^1(\varSigma^2\mathcal{A})\cong\pi\left(\Omega^1(\varSigma^2\mathcal{A})\right)/\left(\pi(K^1
(\varSigma^2\mathcal{A})\right)+\pi\left(dK^0(\varSigma^2\mathcal{A}))\right)$. But $\pi\left(dK^0(\varSigma^2\mathcal{A})\right)\subseteq\pi
\left(\Omega^1(\mathcal{A}\otimes\mathcal{S})\right)$ and $K^0(\mathbb{C}[z,z^{-1}])=\{0\}$. This says that
\begin{eqnarray*}
&  & \widetilde \Omega_{\varSigma^2 D}^1(\varSigma^2 \mathcal{A})\\
& \cong & \left(\pi(\Omega^1(\mathcal{A}\otimes \mathcal{S}))\oplus \pi(\Omega^1(\mathbb{C}[z,z^{-1}])\right)/\left(\pi(\Omega^1(\mathcal{A}\otimes
\mathcal{S}))\oplus \pi(K^1(\mathbb{C}[z,z^{-1}]))\right)\\
& \cong & \pi\left(\Omega^1(\mathbb{C}[z,z^{-1}])\right)/\pi\left(K^1(\mathbb{C}[z,z^{-1}]))\right)\\
& \cong & \widetilde \Omega_N^1\left(\mathbb{C}[z,z^{-1}])\right)\\
& \cong & \mathbb{C}[z,z^{-1}]\,.
\end{eqnarray*}
Here, the first isomorphism follows from the fact that (see Proposition $(3.8)$ in \cite{cg})
\begin{center}
$\pi\left(\Omega^1(\varSigma^2 \mathcal{A})\right)=\pi(\Omega^1(\mathcal{A}\otimes \mathcal{S}))\bigoplus \pi(\Omega^1(\mathbb{C}[z,z^{-1}])\,$,
\end{center}
and we refer (\cite{cp}) for the following fact
\begin{center}
$\widetilde \Omega_N^n\left(\mathbb{C}[z,z^{-1}]\right)=\begin{cases}
\mathbb{C}[z,z^{-1}];\,\,n=0,1\\
\{0\};\quad otherwise
\end{cases}$.
\end{center}
\end{proof}

\begin{remark}\rm
Recall that $\varSigma^2\mathcal{A}\cong\mathcal{A}\otimes\mathcal{S}\bigoplus\mathbb{C}[z,z^{-1}]\,$ as $\,\mathbb{C}\,$-vector spaces, where
$\mathbb{C}[z,z^{-1}]$ is identified with the quotient $\varSigma^2\mathcal{A}/\mathcal{A}\otimes\mathcal{S}$. These direct sum and isomorphism
are also as $\varSigma^2\mathcal{A}$-bimodule. Hence, $\,\widetilde\Omega_{\varSigma^2 D}^1\left(\varSigma^2\mathcal{A}\right)\,$ is always
finitely generated projective $\varSigma^2\mathcal{A}$-bimodule $($Compare with Assumption $2.13$ in \cite{fgr}, Page $131)$.
\end{remark}

\begin{lemma}\label{to refer in next lemma}
$\oint\pi(\omega)=0$ for any $\,\omega\in\Omega^n(\mathcal{A}\otimes\mathcal{S})$ and for all $n\geq 2$.
\end{lemma}

\begin{proof}
Recall Lemma $(3.15)$ from (\cite{cg}) which says that
\begin{center}
$\pi(\Omega^n(\mathcal{A}\otimes\mathcal{S}))\,\,=\,\,\sum_{r=0}^nF^r\pi(\Omega^{n-r}(\mathcal{A}))\otimes\mathcal{S}$.
\end{center}
Hence, for $\,\omega\in\Omega^n(\mathcal{A}\otimes\mathcal{S})$ we have $\,\pi(\omega)=\sum_{r=0}^n\sum_kF^r\pi(v_{r,k})\otimes T_{r,k}\,$ with
$v_{r,k}\in\Omega^{n-r}(\mathcal{A})$. Writing each $T_{r,k}$ in terms of elementary matrices $(e_{ij})$ we get
\begin{center}
$\pi(\omega)\,\,=\,\,\sum_{r=0}^n\sum_{k,i,j}F^r\pi(v_{r,k}^{ij})\otimes e_{ij}\,$.                                     
\end{center}
Then,
\begin{eqnarray*}
\oint\pi(\omega) & = & \,\lim_{t \rightarrow 0}\,t^{p+1}Tr\left(\pi(\omega)e^{-t|\varSigma^2 D|})\right)\\
& = & \,\sum_{r=0}^n\sum_{k,i}\lim_{t \rightarrow 0}\left(t^pTr\left(F^r\pi(v_{r,k}^{ii})e^{-t|D|}\right)\right)\left(te^{-ti}\right)\\
& = & \,0
\end{eqnarray*}
and we are done.
\end{proof}

\begin{lemma}\label{kernels of trace}
$\pi\left(K^n(\varSigma^2 \mathcal{A})\right)=\pi\left(\Omega^n(\mathcal{A}\otimes\mathcal{S})\right)\bigoplus\pi\left(K^n(\mathbb{C}[z,z^{-1}])
\right)$, for all $\,n\geq 2$.
\end{lemma}

\begin{proof}
Note that for any algebra $\mathcal{A}$, we have
\begin{center}
$\Omega^n(\mathcal{A})\,\,=\,\,\underbrace{\Omega^1(\mathcal{A})\otimes_{\mathcal{A}}\ldots\otimes_{\mathcal{A}}\Omega^1(\mathcal{A})}_{n\,\,times}\,$.
\end{center}
Lemma (\ref{Main calculation}) proves that $\,\pi\left(\Omega^1(\mathcal{A}\otimes\mathcal{S})\right)\subseteq\pi\left(K^1(\varSigma^2 \mathcal{A})
\right)$. Since,
\begin{center}
$\Omega^n(\mathcal{A}\otimes\mathcal{S})=\Omega^{n-1}(\mathcal{A}\otimes\mathcal{S})\bigotimes_{\varSigma^2\mathcal{A}}\Omega^1(\mathcal{A}\otimes
\mathcal{S})\,,$
\end{center}
we get
\begin{center}
$\pi\left(\Omega^n(\mathcal{A}\otimes\mathcal{S})\right)\subseteq\pi\left(K^n(\varSigma^2 \mathcal{A})\right)$
\end{center}
becuase $K^\bullet$ is a graded ideal in $\,\Omega^\bullet$. Hence, we have the inclusion `$\supseteq$'. Now,
recall Proposition $(3.8)$ from (\cite{cg}), which says that
\begin{center}
$\pi\left(\Omega^n(\varSigma^2\mathcal{A})\right)\,\,=\,\,\pi\left(\Omega^n(\mathcal{A}\otimes\mathcal{S})\right)\bigoplus\pi\left(\Omega^n
(\mathbb{C}[z,z^{-1}])\right)\,;\,\,\forall\,n\geq 0\,$.
\end{center}
Since $K^n\subseteq\Omega^n$, using Lemma (\ref{to refer in next lemma}) we get the inclusion `$\subseteq$' and this completes the proof.
\end{proof}

\begin{theorem}\label{forms of Frohlich}
For $\,\left(\varSigma^2\mathcal{A},\varSigma^2 \mathcal{H},\varSigma^2 D\right)$,
\begin{enumerate}
\item $\widetilde \Omega_{\varSigma^2 D}^n\left(\varSigma^2\mathcal{A}\right)=\mathbb{C}[z,z^{-1}]\,,\,\,$ for $\,\,n=0,1;$
\item $\widetilde \Omega_{\varSigma^2 D}^n\left(\varSigma^2\mathcal{A}\right)=0\,,\,\,$ for all $\,\,n\geq 2$.
\end{enumerate}
\end{theorem}
\begin{proof}
Part $(1)$ follows from Lemma (\ref{n=0 part}) and Proposition (\ref{to refer in thm.}). Now, Lemma (\ref{Main calculation}\,,\,\ref{kernels of trace})
shows that for all $n\geq 1$,
\begin{eqnarray}\label{to refer here I}
K^n(\varSigma^2 \mathcal{A})+J_0^n(\varSigma^2 \mathcal{A})\,\,=\,\,\Omega^n(\mathcal{A}\otimes \mathcal{S})+K^n(\mathbb{C}[z,z^{-1}])+
J_0^n(\varSigma^2 \mathcal{A})\,.
\end{eqnarray}
But $\,J_0^n(\varSigma^2 \mathcal{A})\subseteq K^n(\varSigma^2 \mathcal{A})$. Hence, equation (\ref{to refer here I}) reduces to
\begin{eqnarray}\label{to refer here V}
K^n(\varSigma^2 \mathcal{A})\,\,=\,\,\Omega^n(\mathcal{A}\otimes \mathcal{S})+K^n(\mathbb{C}[z,z^{-1}])+J_0^n(\varSigma^2 \mathcal{A})\,.
\end{eqnarray}
So, for all $n\geq 1$,
\begin{center}
$dK^n(\varSigma^2 \mathcal{A})\,\,=\,\,d\Omega^n(\mathcal{A}\otimes \mathcal{S})+dK^n(\mathbb{C}[z,z^{-1}])+dJ_0^n(\varSigma^2 \mathcal{A})\,$;
\end{center}
and consequently for all $\,n\geq 1\,$,
\begin{eqnarray}\label{to refer here II}
\pi\left(dK^n(\varSigma^2 \mathcal{A})\right) & = & \pi\left(d\Omega^n(\mathcal{A}\otimes\mathcal{S})\right)+\pi\left(dK^n(\mathbb{C}[z,z^{-1}])
\right)+\pi(dJ_0^n(\varSigma^2 \mathcal{A}))\,. 
\end{eqnarray}
Recall Proposition $(3.8)$ and $(3.10)$ from (\cite{cg}), which say that
\begin{eqnarray}\label{to refer here III}
\pi\left(\Omega^n(\varSigma^2\mathcal{A})\right)\,\,=\,\,\pi\left(\Omega^n(\mathcal{A}\otimes\mathcal{S})\right)\bigoplus\pi\left(\Omega^n
(\mathbb{C}[z,z^{-1}])\right)\,;\,\,\forall\,n\geq 0\,,
\end{eqnarray}
and
\begin{eqnarray}
\pi\left(dJ_0^n(\varSigma^2\mathcal{A})\right)\,\,=\,\,\pi\left(dJ_0^n(\mathcal{A}\otimes\mathcal{S})\right)\bigoplus\pi\left(dJ_0^n(\mathbb{C}
[z,z^{-1}])\right)\,;\,\,\forall\,n\geq 1\,.
\end{eqnarray}
Hence, equation (\ref{to refer here II}) turns out to be
\begin{eqnarray}\label{to refer here IV}
\pi\left(dK^n(\varSigma^2 \mathcal{A})\right) & = & \pi\left(d\Omega^n(\mathcal{A}\otimes\mathcal{S})\right)\bigoplus\pi\left((dK^n+dJ_0^n)
(\mathbb{C}[z,z^{-1}])\right)\,;\,\,\forall\,n\geq 1\,.
\end{eqnarray}
Finally, using equations (\ref{to refer here V}\,,\,\ref{to refer here III}\,,\,\ref{to refer here IV}) we have for all $n\geq 2$,
\begin{eqnarray*}
\widetilde\Omega_{\varSigma^2 D}^n\left(\varSigma^2\mathcal{A}\right) & \cong & \frac{\pi\left(\Omega^n(\varSigma^2\mathcal{A})\right)}{\pi
\left(K^n(\varSigma^2\mathcal{A})\right)+\pi\left(dK^{n-1}(\varSigma^2\mathcal{A})\right)}\\
& \cong & \frac{\pi\left(\Omega^n(\mathcal{A}\otimes\mathcal{S})\right)\bigoplus\pi\left(\Omega^n(\mathbb{C}[z,z^{-1}])\right)}{\pi
\left(\Omega^n(\mathcal{A}\otimes\mathcal{S})\right)\bigoplus\pi\left((K^n+dK^{n-1}+dJ_0^{n-1})(\mathbb{C}[z,z^{-1}])\right)}\\
& \cong & \frac{\pi\left(\Omega^n(\mathbb{C}[z,z^{-1}])\right)}{\pi\left((K^n+dK^{n-1}+dJ_0^{n-1})(\mathbb{C}[z,z^{-1}])\right)}
\end{eqnarray*}
Now, the facts that $\,\pi_N\left(\Omega^n(\mathbb{C}[z,z^{-1}])\right)=\mathbb{C}[z,z^{-1}]$ and $\,\pi_N(dJ_0^{n-1}(\mathbb{C}[z,z^{-1}]))=
\mathbb{C}[z,z^{-1}]$ for all $n\geq 2$ (see Lemma $[3.11]$ and $[3.12]$ in \cite{cg}) completes Part $(2 )$.
\end{proof}

In view of Theorem (\ref{final thm}) and (\ref{forms of Frohlich}), our conclusion of this article comes as the following final theorem.
\begin{theorem}\label{final thm.}
There is a category $\,\mathcal{S}pec$ of spectral triples such that the Dirac dga, the FGR dga and the quantum double suspension, denoted by
$\,\mathcal{F},\mathcal{G},\varSigma^2$ respectively, become covariant functors. Let $\,\mathcal{C}$ be the subcategory of commutative spectral
triples. Restricted to $\mathcal{C}$ both the functor $\mathcal{F}$ and $\mathcal{G}$ are equal to the de-Rham dga. Unlike $\mathcal{F}\circ
\varSigma^2$, the fucntor $\mathcal{G}\circ\varSigma^2$ becomes a constant functor on $\mathcal{S}pec$.
\end{theorem}
\bigskip

\section*{Acknowledgement}
Satyajit Guin gratefully acknowledges financial support of DST, India through INSPIRE Faculty award (Award No. DST/INSPIRE/04/2015/000901).

\end{document}